\documentclass[letterpaper, 10 pt, conference]{ieeeconf}  

\IEEEoverridecommandlockouts                              

\usepackage{graphicx}      
\usepackage{amsmath,amssymb,amsopn,cases}
\usepackage{hhline}
\usepackage{blindtext}
\usepackage{multirow}

\newcommand{\supeq}{\geqslant}
\newcommand{\infeq}{\leqslant}

\DeclareMathOperator*{\Argmin}{Argmin}
\DeclareMathOperator*{\tr}{tr}

\newtheorem{theo}{Theorem}

\newtheorem{coro}{Corollary}
\newtheorem{lemma}{Lemma}
\newtheorem{remark}{Remark}
\newtheorem{example}{Example}

\title{\LARGE \bf Static State and Output Feedback Synthesis for Time-Delay Systems*} 

\author{Matthieu Barreau$^{1}$, Fr{\'e}d{\'e}ric Gouaisbaut$^{1}$ and Alexandre Seuret$^{1}$
\thanks{*This work is supported by the ANR project SCIDiS contract number 15-CE23-0014.}
\thanks{$^{1}$ LAAS-CNRS, Universit\'e de Toulouse, CNRS, UPS, Toulouse, France. (e-mail: barreau,seuret,fgouaisb@laas.fr).}%
}

\begin{document}

\maketitle
\thispagestyle{empty}
\pagestyle{empty}

\begin{abstract}
In this paper, the design of a static feedback gain for a linear system subject to an input delay is studied. This synthesis is based on a stability analysis conducted using  Lyapunov-Krasovskii theorem and Bessel-Legendre inequalities expressed in terms of LMIs. Some bilinear non-convex matrix inequalities are obtained to go from analysis to synthesis. They are then difficult to solve and an iterative LMI procedure is given which takes advantage of the elimination lemma. Naturally, slack variables are introduced and then, following an optimization process, values for them are proposed to reduce the conservatism. The two main corollaries discuss the static state and output feedback synthesis. Finally, a comparison is proposed and shows that this formulation introduces small conservatism. 
\end{abstract}


\section{Introduction}
The analysis and the control of time-delay systems is a very active research area since several decades and this for both practical and theoretical reasons. Indeed, delays are  very common phenomena in control loops, modeling, for instance, the transport time in communication networks \cite{Zampieri:IFAC:2008}. It is also a simple way to model complex dynamics \cite{opac-b1100602} embedding those phenomena into a rough and maybe unprecise delay. From a theoretical point of view, the introduction of a delay in the control loop leads to an infinite dimensional system whose properties such as stability and performances remain complicated to study \cite{norm}. The effects of the delay on the closed loop system properties range from a classical degradation of the performances to instability. Surprisingly, one may also obtain stability by adding delays \cite{opac-b1100602,norm}. The problem of designing a control law for stabilizing a delay system thus remains an open problem.

Roughly speaking, one can classify the control laws in two categories. The first one considers infinite dimensional controllers, taking into account finally the infinite nature of the closed loop. The idea, the so-called Smith predictor, was originally proposed by Smith and also Osipov in a state space approach (see for instance the interesting survey \cite{Palmor}) . The main idea is to design of controller in order to pull out the delay of the closed loop. The stability of the closed loop is therefore ensured using classical tools. Recently, we assist to a renewal of this technique, where the delay is viewed as a transport equation and the predictor is re-interpreted as a backstepping method \cite{krstic}. Notice that this approach has been then extended to a wide class of PDEs \cite{krstic}.

Nevertheless, the main problem of the approach is that it is in general mandatory to know perfectly the delay even if some papers consider the mismatch between the delay and its value used in the control law. The implementation is also a complicated problem \cite{V-AssDamLafRich:CDC:1999}. 

Another idea is closer to the robust approach. In that case, the delay is viewed as a perturbation element and a classical state feedback $u(t)=Kx(t)$ or output state feedback $u(t)=KCx(t)$ is designed to robustly stabilize the closed loop system despite the delay.  Apparently less complex than the first type of controller, this approach was very popular a decade ago. Basically, it was relying on the choice of an adequate Lyapunov-Krasovskii functional and some optimization scheme to be solved. An optimized choice of the Lyapunov parameters allows designing a controller gain with a larger delay upper bound. Nevertheless, all the optimization schemes proposed fall into the Bilinear Matrix Inequality (BMI) category, well-known to be difficult to solve \cite{FRISEURIC:AUT:04} since the controller gain is assumed to be unknown. Generally, the approaches proposed to tackle this problem are always based on the same idea: fixing some parameters of the BMI and then obtain an LMI condition. Some griding techniques may be also adopted but it generally results in very conservative results \cite{li1997delay, fridman2002improved,7045593}. This general heuristic procedure may be sometimes improved by considering some slack variables even if, at the end, one faces the same problem \cite{zhang2005delay,CHEN20061067}.

In this paper, we focus on the design of an output state feedback without any information on the delay value. Considering the Lyapunov functional proposed by \cite{seuret:hal-01065142}, a Bilinear Matrix Inequality is proposed. An extensive use of the elimination lemma and an optimized choice of some slack variables allow reducing the optimisation problem into a set of LMIs. At a price of an increasing conservatism, one obtains an optimisation procedure with a limited numerical burden.

\textbf{Notations.} Throughout the paper, $\mathbb{R}^n$ stands for the $n$ dimensional Euclidian space, $\mathbb{R}^{n \times m}$ for the set of all $n \times m$ matrices. $\mathbb{S}^n$ is the subset of $\mathbb{R}^{n \times n}$ of symmetric matrices such that $P \in \mathbb{S}_+^n$ or equivalently $P \succ 0$ denotes a symmetric positive definite matrix.  For any square matrices $A$ and $B$, three operations are defined as follow: $\text{He}(A) = A + A^{\top}$, $\tr$ is the trace operator and $\text{diag}(A,B) = \left[ \begin{smallmatrix}A & 0\\ 0 & B \end{smallmatrix} \right]$. The notations $I_n$ and $0_{n, m}$ denote the $n$ by $n$ identity matrix and the null matrix of size $n  \times m$. The state variable $x$ can be represented using the Shimanov notation \cite{kolmanovskii}:
\[
	x_t: \left\{ \begin{array}{ccl}
		[-h, 0] & \to & \mathbb{R}^n\\
		\tau    &  \mapsto & x(t+\tau)
	\end{array}
	\right.
\]

\section{Problem Statement}

Throughout this paper, we are interested in the design of static state and output feedback controller gains for the following system subject to an input delay:
\begin{equation} \label{eq:problem}
	\left\{
		\begin{array}{ll}
			\dot{x}(t) = Ax(t) + Bu(t-h), & \quad \forall t > 0,\\
			y(t) = Cx(t), & \quad \forall t > 0, \\
			u(t) = Ky(t), & \quad \forall t > 0, \\
			x(t) = \phi(t), & \quad \forall t \in [-h, 0],
		\end{array}
	\right.
\end{equation}
with $h > 0$ constant and known, $x(t) \in \mathbb{R}^n$, $A \in \mathbb{R}^{n \times n}, B \in \mathbb{R}^{n \times m}, C \in \mathbb{R}^{p \times n}$ and $K \in \mathbb{R}^{m \times p}$. The initial condition is $\phi \in \mathcal{C} ([-h, 0], \mathbb{R}^n)$. It is well known from Theorem 2.1 in \cite{kolmanovskii} that with these conditions, there is a unique solution to this system on $\mathbb{R}^+$.

This system can be largely encountered in practical situations. For example, as in the case of a regenerative chatter, a vibration absorber or a network controlled system as explained in \cite{opac-b1100602,OLGAC199493}.

The purpose of this article is to adopt efficient stability analysis of time-delay systems in terms of Linear Matrix Inequalities (LMI) to synthesize the controller gain $K$ such that system \eqref{eq:problem} becomes asymptotically stable. If $\text{rank}(C) = n$, this problem is known as ``Static State Feedback'' (SSF), and if $\text{rank}(C) < n$ this is a ``Static Output Feedback'' (SOF).

\section{Preliminary Results}

In this section, two technical tools used to assess stability of system \eqref{eq:problem} are introduced.

\subsection{The Elimination Lemma}

The conditions will be expressed in terms of Linear Matrix Inequality (LMI) and especially of this form:
\begin{equation} \label{eq:LMIgen}
	\exists Q \in \mathbb{S}^{n \mathcal N}\!\!\!, \ \ \ x^{\top} Q x \leq 0, \ \ \ \forall x \in \mathbb{R}^{n \mathcal N}\!\!\!, \quad \text{s.t.} M x = 0, 
\end{equation}
with an integer $\mathcal N > 1$ and $M = \left[ \begin{matrix} A_1 & A_2 & \cdots & A_{\mathcal N} \end{matrix} \right]$ of appropriate dimension. In other words, that means $Q$ is a negative semi-definite matrix on the kernel of matrix $M$.

To transform this LMI into a semi-definite programming formulation, we can use the elimination lemma.
\begin{lemma}[Elimination Lemma \cite{Svariable}] \label{sec:finsler}
	For any $Q \in \mathbb{S}^{n \mathcal{N} }$ and $M \in \mathbb{R}^{n \times \mathcal{N} n}$, the following statements are equivalent:
	\begin{enumerate}	
		\item $x^{\top} Q x < 0$ for all $x \in \mathbb{R}^{n \mathcal N} \text{ such that } Mx = 0$,
		\item $\exists Y \in \mathbb{R}^{n \times n \mathcal{N}}, Q + \text{He} \left( M^{\top} Y \right) \prec 0$,
		\item ${M^{\perp}}^{\top} Q M^{\perp} \prec 0$ where $MM^{\perp} = 0$.
	\end{enumerate}
\end{lemma}
\begin{proof}
	 The main idea of the proof for (1) $\Rightarrow$ (2) is reminded for a better understanding of the paper.  Assume (1) holds. Then, according to Finsler's lemma (developed in \cite{finsler}), $\exists \sigma \in \mathbb{R}, Q - \sigma M^{\top} M \prec 0$. Then, taking $Y = \frac{\sigma}{2}M$ with $\sigma \in \mathbb{R}$ in (2) concludes the proof. That means there always exists $Y$ belonging to $\text{span}(M)$.
\end{proof}

The main advantage of this formulation is that it separates the LMI into two parts: the first one is related to the stability of an extended system and the other one is related to the relation between the extra-states.

\subsection{Stability analysis: a hierarchy of LMIs}

Consider the following time-delay system for $t > 0$:
\begin{equation} \label{eq:modifiedProblem}
	\dot{x}(t) = Ax(t) + A_d x(t-h),
\end{equation}
with $A$ and $A_d$ known matrices of appropriate dimension and $\phi$ is initial condition on $[-h, 0]$. There exists a hierarchy of LMI conditions derived using a Lyapunov functional and Bessel inequality. The result obtained is Theorem 5 in \cite{seuret:hal-01065142} and a slightly modified version is presented below:
\begin{theo} \label{sec:theo}
	For a given integer $N > 0$ and a constant delay $h > 0$, if there exist matrices $P_N \in \mathbb{S}^{(N+1)n}_ +$, $S, R \in \mathbb{S}^n_ +$ such that the following equation
	\begin{equation} \label{eq:LMI1}
		\forall \xi \in \mathbb{R}^{n(N+3)} \text{ s.t. } M \xi = 0, \quad \xi^T \Phi_N \xi < 0, \\
	\end{equation}
	holds with
	\vspace{-0.5cm}
	\[
		\!\begin{array}{l}
			\Phi_N = \Phi_{N0} - \left[ \begin{smallmatrix} \Gamma_N(0) \\ \vdots \\ \Gamma_N(N) \end{smallmatrix} \right]^{\top} R_N \left[ \begin{smallmatrix} \Gamma_N(0) \\ \vdots \\ \Gamma_N(N) \end{smallmatrix} \right], \\
			\Phi_{N0} = \text{He} \left( G_N^{\top} P_N H_N \right) + \tilde{S}_N + h^2 F_N^{\top} R F_N, \\
			\tilde{S}_N = \text{diag}\left(0_n, S, -S, 0_{Nn} \right), \\
			 R_N = \text{diag} \left( R, 3R, \cdots, (2N+1)R \right), \\
			 F_N = \left[ \begin{matrix} I_n & 0_n & 0_n & 0_{n, nN} \end{matrix} \right], \\
			 G_N = \left[ \begin{matrix} 0_n & I_n & 0_n & 0_{n,nN} \\ 0_{nN, n} & 0_{nN, n} & 0_{nN, n} & h I_{nN} \end{matrix} \right], \\
			 H_N = \left[ \begin{matrix} F_N^{\top} & \Gamma_N^{\top}(0) & \cdots & \Gamma_N^{\top}(N-1) \end{matrix} \right]^{\top}, \\
			 \Gamma_N(k) = \left[ \begin{matrix} 0_n & I_n & (-1)^{k+1} I_n & \gamma_{Nk}^0 I_n & \cdots & \gamma_{Nk}^{N-1} I_n \end{matrix} \right], \\
			 \gamma_{Nk}^i = \left\{ \begin{array}{ll} -(2i+1) \left( 1 - (-1)^{k+i} \right), & \quad \text{ if } i \infeq k, \\ 0, & \quad \text{ else,} \end{array} \right.\\
			 M = \left[ \begin{matrix} I_n & -A & -A_d & 0_{n, nN} \end{matrix} \right], \\
		\end{array}
	\]
	then system \eqref{eq:modifiedProblem} is asymptotically stable.
\end{theo}
 \begin{proof} This result mostly comes from the proof of Theorem 5 in \cite{seuret:hal-01065142}, the LMI negativity condition from this theorem is denoted $\tilde{\Phi}_N \prec 0$. The positivity condition has been replaced by the more simple but slightly more conservative one $P_N \succ 0$. Let $M^{\perp}$ defined as follows: $M^{\perp} = \left[ \begin{smallmatrix} A & A_d & 0_n \\ & I_{(N+2)n} & \end{smallmatrix} \right]$.
We notice that $\Phi_N = {M^{\perp}}^{\top} \tilde{\Phi}_N M^{\perp}$ and $M M^{\perp} = 0_{n, 3n}$. Applying the equivalence between Proposition (3) and (1) of Lemma \ref{sec:finsler} ensures the equivalence between the two LMIs. Then system \eqref{eq:modifiedProblem} is indeed asymptotically stable. 
 \end{proof}

\begin{remark} For the proof of Theorem 5 in \cite{seuret:hal-01065142}, the extended state used is:
\[
	\tilde{\xi}( x_t ) = \left[  \begin{matrix} x^{\top}_t (0) & x_t^{\top}(-h) & \frac{1}{h} \Omega_0^{\top} & \cdots & \frac{1}{h} \Omega_{N-1}^{\top} \end{matrix} \right]^{\top},
\]
where $\Omega_i$ is the $i^{\text{th}}$ projection of $x_t$ over a basis of Legendre polynomials on $[-h, 0]$. The Legendre polynomials are not used in the sequel. The extended state $\xi$ in this revised version is as follows:
\[
	\xi( x_t, \dot{x}_t )\! =\! \left[  \begin{matrix} \dot{x}_t^{\top}(0) &\!\! x^{\top}_t (0) &\!\! x_t^{\top}(-h) &\!\! \frac{1}{h} \Omega_0^{\top} & \!\!\dots\! &\! \frac{1}{h} \Omega_{N-1}^{\top} \end{matrix} \right]^{\top}\!\!\!\!\!.
\]
That leads to $\xi = M^{\perp} \tilde{\xi}$ and $M\xi = 0$. 
\end{remark}
\begin{remark} This manipulation can also be interpreted for limited values of $N$ to weighted method of \cite{7045593}.
\end{remark}
These LMI conditions are very effective as for $N = 1$ this is related to Wirtinger-based inequality from \cite{wirtinger}. For $N \supeq 2$, this theorem encompasses Wirtinger's inequality. The conservatism induced by the choice of the Lyapunov functional and Bessel's inequality is reduced as $N$ increases. This makes a very efficient tool for stability analysis. Applying Theorem \ref{sec:theo} to system \eqref{eq:problem} can be performed considering $A_d = BKC$. Using the equivalence between propositions (1) and (2) of Lemma \ref{sec:finsler} on equation \eqref{eq:LMI1} leads to:
\begin{equation}  \label{eq:LMImodified}
	\exists Y \in \mathbb{R}^{n \times n(N+3)}, \quad \quad \Phi_N + \text{He} \left( M^{\top} Y \right) \prec 0,
\end{equation}
with $M = \left[ \begin{matrix} I_n & \ -A & \ -BKC & \ 0_{n,nN} \end{matrix} \right]$. This is still not an LMI but the $K$ matrix is only present in the second term of the inequality, making the synthesis easier as it is possible to see in the next section. 

\section{Variations on the Elimination Lemma}

Using the matrix inequalities obtained in Theorem \ref{sec:theo} with an unknown gain $K$ is not a linear problem. This results in numerically non-tractable inequalities. The following lemma is therefore introduced in order to linearize the problem.

\begin{lemma} \label{sec:finsler2}
	For any $Q \in \mathbb{S}^{n \mathcal N}$, $M \in \mathbb{R}^{n \times n \mathcal N}$, $\mathcal{I} \in \mathbb{R}^{n \times n}$ non singular and  $(\varepsilon_1,  \dots, \varepsilon_{\mathcal N}) \in \mathbb{R}^{\mathcal N}$, the following holds:
		
		If there exists $Z \in \mathbb{R}^{n \times n}$ such that 
			\[
				\begin{array}{c}
					Q + \text{He} \left( M^{\top} Z F_{\varepsilon} \right) \prec 0, \\
					\text{with } F_{\varepsilon} = \left[ \begin{matrix} \varepsilon_1 \mathcal{I} & \varepsilon_2 \mathcal{I} & \cdots & \varepsilon_{\mathcal N} \mathcal{I} \end{matrix} \right];
				\end{array}
			\] 
			
		Then $x^{\top} Q x < 0$ for all $x \in \ker(M)$.
\end{lemma}
\begin{proof}The proof of Lemma \ref{sec:finsler} shows that there exists a $Y$ which belongs to a subspace $\mathcal{U}$ of $\mathbb{R}^{nN \times n}$ with $\text{Span}(M) \subset \mathcal{U}$. Here, we assume $Y$ to be in the set $\mathcal{U}_{\varepsilon} = \left\{ Z F_{\varepsilon} \ | \ Z \in \mathbb{R}^{n \times n} \right\}$. As there is no guaranty that $\mathcal{U} \cap \mathcal{U}_{\varepsilon} \neq \left\{ 0 \right\}$, that leads to this more conservative lemma. For synthesis purposes detailed later on, assume that $Z$ defined previously and $\mathcal{I}$ commutes. As $Z \in \mathbb{R}^{n \times n}$ is unknown, $\mathcal{I}$ must commute with all the matrices, and then, $\mathcal{I} = I_n$ is required.\end{proof}


From now on, the most difficult part is to find values for $\varepsilon_i$ such that the result is the ``least conservative''. The idea developed here is to minimize the distance between the subspace $\mathcal{U}$ and $\mathcal{U_{\varepsilon}}$. Then, as there exists $Y \in \mathcal{U}$ such that equation $Q + \text{He} \left( M^{\top} Y \right) \prec 0$ is satisfied, there surely exists a neighborhood around $Y = \frac{\sigma}{2}M$ in which this is still true. If the two subspaces $\mathcal{U}$ and $\mathcal{U}_{\varepsilon}$ are close enough, the intersection between this neighborhood and $\mathcal{U}_{\varepsilon}$ may not be empty.
To do so, a distance between the two subspaces needs to be introduced to correctly set up the optimization process.

\subsection{Optimization over $\varepsilon_1, \dots \varepsilon_N$}

In this subsection, we consider without loss of generality that $M$ has the following structure:
\[
	M = \left[ \begin{matrix} A_1 & A_2 & \cdots & A_{\mathcal{N}} \end{matrix} \right],
\]
with $A_1, \cdots, A_{\mathcal N} \in \mathbb{R}^{n \times n}$.
The Frobenius norm on the space of matrices is defined as follow: 
\[
	\forall U \in \mathbb{R}^{n \mathcal{N} \times n}, \quad \| U \|_F^2 = \tr \left( U U^{\top} \right).
\]
The problem is then to find the values $\varepsilon_i$ such that $\| M - F_{\varepsilon} \|_F^2$ is minimized: 
\[
	\varepsilon^{opt} = \left( \varepsilon^{opt}_1, \varepsilon^{opt}_2, \cdots, \varepsilon^{opt}_{\mathcal N} \right) = \Argmin_{\varepsilon_1, \dots, \varepsilon_{\mathcal N} \in \mathbb{R}} \left(\| M - F_{\varepsilon}\|_F^2 \right).
\]
Using properties of the trace operator, one gets:
	\[
		\begin{array}{ll}
			\varepsilon^{opt} \!\!\!\!& \displaystyle   = \Argmin_{\varepsilon_1, \dots, \varepsilon_{\mathcal N} \in \mathbb{R}} \tr \left( - \sum_{i=1}^{\mathcal N} \varepsilon_i \left( \mathcal{I} A_i^{\top} + A_i \mathcal{I}^{\top} \right) + \varepsilon_i^2 \mathcal{I} \mathcal{I}^{\top} \right) \\
			& \displaystyle = \Argmin_{\varepsilon_1, \dots, \varepsilon_{\mathcal N} \in \mathbb{R}} \sum_{i=1}^{\mathcal N} \underbrace{\left( -  2 \varepsilon_i \tr \left( \mathcal{I} A_i^{\top}\right) + \varepsilon_i^2 \tr \left( \mathcal{I} \mathcal{I}^{\top} \right) \right)}_{f_i(\varepsilon_i)} \\
			& \displaystyle = \Argmin_{\varepsilon_1, \dots, \varepsilon_{\mathcal N} \in \mathbb{R}} g(\varepsilon).
		\end{array}
	\]
First, notice that each $f_i$ is convex so $g$ is convex. To find the argument for the minimum, a first idea is to compute the Hessian of $g$: $\nabla^2g = \tr(\mathcal{I} \mathcal{I}^{\top}) I_{\mathcal N}$. As $\mathcal{I}$ is invertible, $\tr(\mathcal{I} \mathcal{I}^{\top}) > 0$, so $\nabla^2g$ is definite positive and it has a global minimum. To find the argument for the minimum, we need to solve $\nabla g (\varepsilon^{opt}) = 0$, or in other words: $-2 \tr \left( \mathcal{I} A_i^{\top}\right)  + 2 \varepsilon^{opt}_i = 0$ for all $i \in [1, \mathcal N]$. That leads to the following:
\[
	\forall i \in [1, \mathcal N], \quad \varepsilon_i^{opt} = \frac{\tr(\mathcal{I} A_i^{\top})}{\tr(\mathcal{I} \mathcal{I}^{\top})}.
\]

\begin{remark} It is possible to adapt this procedure to another norm and find other optimal values; for example, the induced norm on a matrix $U \in \mathbb{R}^{n \mathcal{N} \times n}, \| U \|_{2}^2 =  \sup_{ \substack{x \in \mathbb{R}^{nN} \\ x \neq 0} } \frac{\| Ux \|^2}{\|x\|^2}$. 
This optimal condition is nevertheless not tractable.
\end{remark}
Then, based on Lemma \ref{sec:finsler2}, we get the following result:
\begin{lemma} \label{sec:finsler3}
	For any $Q \in \mathbb{S}^{n \mathcal N}$, $M \in \mathbb{R}^{n \times n \mathcal N}$, $\mathcal{I} \in \mathbb{R}^{n \times n}$ non singular, the following implication holds: 	
	
	If there exists $Z \in \mathbb{R}^{n \times n}$ such that 
		\[
			\begin{array}{c}
				Q + \text{He} \left( M^{\top} Z F_Z \right) \prec 0,\\
				\text{with } F_Z = \left[ \begin{matrix} \tr(\mathcal{I} A_1^{\top}) \mathcal{I} & \cdots  & \tr(\mathcal{I} A_{\mathcal N}^{\top}) \mathcal{I} \end{matrix} \right];
			\end{array}
		\]
		
		Then $x^{\top} Q x < 0$ for all $x \in \ker(M)$.
\end{lemma}

The optimization made earlier means that there is a $Z \in \mathbb{R}^{n \times n}$ such that $\tr(A_i) Z$ is as close as possible to $A_i$ and it must be true for all $i \in [1, \mathcal{N}]$. If the eigenvalues of a matrix $A_i$ are ``far'' from each-other, then this method is not efficient and the conservatism grows as $n$ and $\mathcal N$ increase.\\
Indeed, if $\tr(A_1) = 0$ for example, following our methodology, the distance between $\varepsilon_1 I_n$ and $A_1$ is minimized for $\varepsilon_1 = 0$ even if $A_1$ is far from being the null matrix. To correct this phenomenon, we propose an alternative methodology based on the same optimization process.

\subsection{Towards a structured $\mathcal{I}$}
\label{sec:jordan}
The idea is to give a particular structure to $\mathcal{I}$ and $Z$. As we would like to perfectly fit between $A_{\ell}$ and $\mathcal{I}Z$ for a given $\ell \in [1, \mathcal{N}]$, one solution is to deal with each eigenvalue of a matrix $A_{\ell}$ with a different $\varepsilon$. And to do so, the matrix $A_{\ell}$ can be transformed into its real Jordan form. Then one gets:
\begin{equation} \label{eq:jordanForm}
	\begin{array}{l}
		\tilde{A}_i = T^{-1} A_i T, \\
		\tilde{A}_{\ell} = \text{diag} \left( \tilde{A}_{\ell}(1), \dots, \tilde{A}_{\ell}(k) \right), \\
		\forall j \in [1, k], \quad J_j \tilde{A}_{\ell} J_j^{\top} = \tilde{A}_{\ell}(j), \\
	\end{array}
\end{equation}
where $T$ is a non singular matrix such that $\tilde{A}_{\ell}$ is the Jordan normal form of $A_{\ell}$, that means $\tilde{A}_{\ell}(j)$ is a real Jordan block and there are $k$ different Jordan blocks. If we denote by $r_j$ for $j \in [1, k]$ the algebraic multiplicity of eigenvalue $j$ of $A_{\ell}$, then $J_j$ is a $r_j$ by $n$ matrix of the form $J_j = \left[ \begin{matrix} 0 & I_{r_j} & 0 \end{matrix} \right]$.\\
Instead of the methodology developed above, other slack variables are defined:
\begin{equation} \label{eq:notation}
	\begin{array}{l}
		F_W = \left[ \begin{matrix} \mathcal{E}_1 & \dots  & \mathcal{E}_N \end{matrix} \right], \quad W = \sum_{j=1}^k  J_j^{\top} W_j J_j, \\
		\forall i \in [1, \mathcal N], \quad \mathcal{E}_i = \sum_{j=1}^k \varepsilon_i(j) J_j^{\top} J_j,\\
	\end{array}
\end{equation}
with $W_j$ a real matrix of same size than $\tilde{A}_j$. Then, to each eigenvalue $\lambda_i$ of $A_{\ell}$, we associate $\varepsilon_{\ell}(i)$.

\begin{remark} Note that $\mathcal{E}_i$ and $W$ commute.\end{remark}

To be clearer with these notations, here is an example.
\begin{example} Assume here $\ell = 1$, $\mathcal N = 2$ and consider:\\
\[
	A_1 = \left[\begin{smallmatrix} 5 & 3 & 3 & 2 \\ 0 & 2 & -2 & -2 \\ -1 & -1 & 3 & 0 \\ 1 & 1 & 1 & 4 \end{smallmatrix} \right] \text{ and } A_2 = \text{diag}\left(I_2, 3 I_2 \right).
\]
We then get $\tilde{A}_1 = \text{diag}\left( \left[ \begin{smallmatrix} 4 & 1 \\ 0 & 4 \end{smallmatrix} \right], 4, 2 \right)$ and $\tilde{A}_2 = \text{diag} \left(  \left[ \begin{smallmatrix} 3 & 0 \\ -2 & 1 \end{smallmatrix} \right],  \left[ \begin{smallmatrix} 3 & 0 \\ 2 & 1 \end{smallmatrix} \right] \right)$. As there are two different eigenvalues in $A_1$, $k=2$ and $r_1 = 2$, $r_2 = 1$.
\end{example}

An adaptation of Lemma \ref{sec:finsler2} can then be derived:
\begin{lemma} \label{sec:finsler4}
	For any $Q \in \mathbb{S}^{n \mathcal N}$, $M \in \mathbb{R}^{n \times n \mathcal N}$ and \\ $(\varepsilon_1(1), \dots, \varepsilon_1(k),  \dots, \varepsilon_{\mathcal N}(k)) \in \mathbb{R}^{k \mathcal{N}}$, the following implication holds: 	
	
	If there exists $W_j \in \mathbb{R}^{r_j \times r_j}$ such that for $F_W$ and $W$ defined in Equation \eqref{eq:notation}:
		\[
			\begin{array}{c}
				Q + \text{He} \left( \tilde{M}^{\top} W F_W \right) \prec 0,\\
				\text{with } \tilde{M} = \left[\begin{matrix} \tilde{A}_1 & \cdots & \tilde{A}_{\mathcal N} \end{matrix} \right];
			\end{array}
		\]
	
	Then $x^{\top} Q x < 0$ for all $x \in \ker(\tilde{M})$.
\end{lemma}

The same optimization process than in the previous part, i.e. find $\varepsilon_i(j)$ such that $\|\tilde{M} - F_W\|_F$ is minimized, leads to:
\begin{equation}
	\forall i \in [1, \mathcal N], \forall j \in [1, k], \quad \varepsilon_i^{opt}(j) = \frac{1}{r_j} \tr \left( J_j \tilde{A}_i J_j^{\top} \right).
\end{equation}

With this process, the eigenvalues of $A_{\ell}$ are treated with different $\varepsilon_i(j)$. So it can deal with spread eigenvalues. Then, the accuracy of this optimization will fall as $\mathcal N$ increases but is independent of $n$. Moreover, the number of slack variables is lower than in the previous case because $W$ is structured.

{\it Example 1 (cont.)} Using matrices from equation \eqref{eq:notation} on Example 1 for $W_1 \in \mathbb{R}^{3 \times 3}$ and $W_2 \in \mathbb{R}$, the optimization results in:
	\[
		 \begin{array}{l}
		 	\mathcal{E}_1 = \text{diag}\left( \varepsilon_1(1) I_3, \varepsilon_1(2) \right) = \text{diag}\left( 4 I_3, 2 \right), \\
			\mathcal{E}_2 = \text{diag}\left( \varepsilon_2(1) I_3, \varepsilon_2(2) \right) = \text{diag}\left( \frac{7}{3} I_3, 1 \right),\\
			W = \text{diag}(W_1, W_2).
		\end{array}
	\]

\begin{remark} This methodology is also effective if $A_{\ell}$ is not in Jordan form. This form was chosen to prevent from a spreading effect. If one chooses for example $J_i = I_n$ for all $i$, then we get the result of the previous subpart. It is nevertheless known that the numerical procedure to get $A_{\ell}$ in its Jordan normal form is not always accurate \cite{golub1976ill}. Nevertheless, Finsler's lemma has been widely used to deal with robustness issues and then the problem can be slightly modified to meet robustness properties also \cite{Svariable}. \end{remark}

With this corollary, a new formulation of the elimination lemma is introduced. It is more conservative but the slack variables $F_W$ have been chosen to minimize its conservatism. This is the main novelty of this paper and notice that it can be applied to many other problems than the one discussed in this paper as: stabilization of sampled, event-triggered or delay-varying systems.

\section{Static State and Output Feedback Synthesis}

We focus now on the design of a controller for time-delay system \eqref{eq:problem}. The feedback gain $K$ is a decision variable and the aim of this part is to convert inequality \eqref{eq:LMImodified} into a LMI. Applying the methodology of the previous section, the following corollary is obtained:
\begin{coro} \label{sec:coro}
	For a given integer $N > 0$ and a constant delay $h > 0$, if there exist matrices $P_N \in \mathbb{S}^{(N+1)n}_ +$, $S, R \in \mathbb{S}^n_ +$ such that the following inequality
	\begin{equation} \label{eq:LMI2}
		\Phi_N + \text{He} \left( \tilde{M} ^{\top} W F_W \right) \prec 0
	\end{equation}
	holds for $\Phi_N$ defined in Theorem \ref{sec:theo}, $F_W$ and $W$ defined in \eqref{eq:jordanForm}-\eqref{eq:notation} for $\ell = 2$ and
	\[
		\tilde{M} = \left[ \begin{matrix} I_n & \ -\tilde{A} & \ - T^{-1} B K C T & \ 0_{n, nN} \end{matrix} \right], \\
	\]
	then system \eqref{eq:modifiedProblem} is asymptotically stable and $W$ is not singular.
\end{coro}
\begin{proof} The proof of stability comes naturally using Lemma \ref{sec:finsler4}, Theorem \ref{sec:theo} and equation \eqref{eq:LMImodified}. $W$ is not singular because inequality \eqref{eq:LMI2} implies that $\text{He}\left(W\right) \prec 0$.
\end{proof}

Inequality \eqref{eq:LMI2} is a Bilinear Matrix Inequality (BMI) and is, as noted in \cite{vanantwerp2000tutorial}, difficult to solve. In the following subsections, we will set up a sequence of LMI to solve in order to give a solution to the BMI problem. The SSF and SOF synthesis are then discussed.

\subsection{Static State Feedback synthesis}

Corollary \ref{sec:coro} already transformed the initial non linear problem stated in Theorem \ref{sec:theo} into a more suitable form for the synthesis of $K$. 
The following corollary linearizes Corollary \ref{sec:coro} in the case $C = I_n$.

\begin{coro}
	\label{sec:coroSSF}
	For a given integer $N > 0$ and a constant delay $h > 0$, if there exist matrices $P_N \in \mathbb{S}^{(N+1)n}_ +$, $S, R \in \mathbb{S}^n_ +$ and $\bar{K} \in \mathbb{R}^{m \times n}$ such that the following inequality
	\begin{equation} \label{eq:LMI3}
		\Phi_{N} + \text{He}\left( \! \left( \mathcal{M} \tilde{X} + \left[ \begin{matrix} 0_n & 0_n & -\tilde{B} \bar{K} & 0_{n, nN} \end{matrix} \right] \right)^{\!\!\top} \!\! F_W \! \right) \prec 0
	\end{equation}
	holds with the same notations than in Corollary \ref{sec:coro} but:
	 \[
	 	\begin{array}{l}
	 		\mathcal{M} = \left[ \begin{matrix} I_n & -\tilde{A} & 0_n & 0_{n, nN} \end{matrix} \right], \quad \tilde{X} = \text{diag}(\underbrace{X, \cdots ,X}_{N+3}),\\
	 		X = \sum_{j=1}^k J_j^{\top} X_j J_j, \quad  \forall j \in [1, k], X_j \in \mathbb{R}^{r_j \times r_j}, \\
	 		\tilde{A} = T^{-1} A T, \quad \tilde{B} = T^{-1} B,
		\end{array}
	 \]
then time-delay system \eqref{eq:problem} is asymptotically stable with the state feedback gain $K = \bar{K} X^{-1} T^{-1}$.
\end{coro}

\begin{proof}
	In the proof of Corollary \ref{sec:coro}, $W$ is not singular, $X = W^{-1}$ can be introduced. A direct calculation shows that $X = \sum_{j=1}^k J_j^{\top} W_j^{-1} J_j$. Note that
	\[
		\tilde{M} = \mathcal{M} +  \left[ \begin{matrix} 0_n & 0_n & - \tilde{B} KT & 0_{n, nN} \end{matrix} \right].
	\]
	Multiplying inequality \eqref{eq:LMI2} on the left by $\tilde{X}^{\top}$ and on the right by $\tilde{X}$ and noticing that $F_N \tilde{X} = X F_N$, $G_N \tilde{X} =  \text{diag}\left( X, X \right) G_N$ and $F_W \tilde{X} = X F_W$, inequality \eqref{eq:LMI2} becomes inequality \eqref{eq:LMI3} with $\bar{K} = K T X$.
\end{proof}

In the previous theorem, inequality \eqref{eq:LMI3} is a bilinear matrix inequality as $\mathcal{E}_3$ depends on $K$. One solution is to solve an LMI with respect to the unknown variables $P, R, S, \bar{K}$ and $X$ and then update $F_W$ and continue for a given number of steps. This heuristic can be summarized below:
\begin{description}
	\setlength\itemsep{0.5em}
	\item[\hspace{-0.3cm}[Init.]\hspace{-0.6cm}] Set $l = 1$ and initiate with $K^0$;
	\item[\hspace{-0.3cm}[Step $l$, 1] \ ] Compute $F_W^l$ using $K^{l-1}$ and freeze this value. In particular, $\varepsilon_3^{opt}(j) = - \frac{1}{r_j} \tr \left( J_j \tilde{B}K^{l-1}T J_j^{\top} \right)$ for $j = 1, \dots, k$;
	\item[\hspace{-0.3cm}[Step $l$, 2] \ ] Solve LMI \eqref{eq:LMI3}, compute and freeze $K^l$;
	\item[\hspace{-0.3cm}[Term.]\hspace{-0.3cm}] Stop when $l = l_{max}$.
\end{description}

In this algorithm, we seek the gain $K^{\ell_{max}}$ such that the feasibility area obtained by Corollary \ref{sec:coroSSF} is close to the one obtained by Theorem \ref{sec:theo}, or, in other words, the norm $\|M - F_W \|_F^2$ is small.\\
The main problem with this method is that a feedback gain for $l = 0$, $h > 0$ is needed. To do so, assuming $(A, B)$ controllable, a feedback gain $K^0$ such that $A + BK^0$ is Hurwitz can be found. \cite{kolmanovskii} shows that the solution $x$ to system \eqref{eq:problem} depends continuously on the delay so there exists a small enough $h$ such that if the system described by $\dot{x}(t) = (A+BK) x(t)$ is stable, then \eqref{eq:problem} is also stable. One can then compute $F_W^1$ and start the previous algorithm. Since the LMIs are continuous in $h$, for a small enough $\Delta h$, $K^{\ell_{max}}$ is a feedback gain for $h + \Delta h$, then the algorithm can be applied again and step by step, one increases $h$ to the desired value. If it is not possible to find $\Delta h$ such that the LMIs are feasible, the algorithm stops and the maximum allowable delay with this technique has been reached. A similar methodology is widely used when it comes to solve Bilinear Matrix Inequality. It is an adaptation of a path-following method described in \cite{hassibi1999path} for example. Of course, this algorithm also introduces some conservatism.

The feedback gains obtained using this algorithm are very sensitive to its initial condition. That is the reason why it is important to know how to initialize the algorithm. This heuristic provides one solution but other initialization techniques can be used as the one described in \cite{arzelier2010mixed} with a random algorithm. This technique has the main advantage that it tries to find the best starting point according to a given criterion.

\begin{remark} The methodology described here can be used without important modification to the SSF synthesis of system:
\[
	\dot{x}(t) = A_0 x(t) + A_1x(t-h) + B u(t-h).
\]
Indeed, a slight change in $\mathcal{M}$ is needed:
\[
	\mathcal{M} = \left[ \begin{matrix} I_n & -\tilde{A}_0 & -\tilde{A}_1 & 0_{n, nN} \end{matrix} \right],
\]
and in the algorithm, $\varepsilon_3 = \frac{-1}{r_j} \tr \left( J_j  T^{-1} \left(A_1+ BK \right)T  J_j^{\top} \right)$ for $j$ from $1$ to $k$. 
\end{remark}

\begin{remark} The proposed algorithm gives a special structure to $T^{-1} B K T$. Indeed, it is close to its Jordan form.\end{remark}

\subsection{Static Output Feedback Synthesis}

In this part, $\text{rank} (C)$ is assumed to be less than or equal to $n$. Contrary to the previous part, this is not a full-state feedback anymore, but only part of the state. In this sense, the synthesis can be considered as a static output-feedback synthesis problem. The methodology of the previous part cannot apply directly because $C$ do not commute with $X$. The idea is then to assume that $X = \sigma I_n$ with $\sigma \in \mathbb{R}$, the following corollary is obtained.

\begin{coro}
	\label{sec:coroSOF}
	For a given integer $N > 0$ and a constant delay $h > 0$, if there exist matrices $P_N \in \mathbb{S}^{(N+1)n}_ +$, $S, R \in \mathbb{S}^n_ +, \bar{K} \in \mathbb{R}^{m \times p}$ and $\sigma \in \mathbb{R}$ such that the following inequality
	\begin{equation} \label{eq:LMI4}
		\Phi_{N} + \text{He}\left( \! \left( \sigma \mathcal{M} + \left[ \begin{matrix} 0_n & 0_n & -\tilde{B} \bar{K} \tilde{C} & 0_{n, nN} \end{matrix} \right] \right)^{\!\!\top} \!\! F_W \! \right) \prec 0
	\end{equation}
	holds with the same notations than in Corollary \ref{sec:coroSSF} and $\tilde{C} = CT$, then time-delay system \eqref{eq:problem} is asymptotically stable with the feedback gain $K = \sigma^{-1} \bar{K}$.
\end{coro}
\begin{remark} As before, the LMI condition ensures $\sigma \neq 0$. As the same result than in the previous corollary is applied but with a restriction on $X$, then a more conservative result is expected (see the first example in the following section). 
\end{remark}

At this stage,  a sequence of LMIs to be solved using the same algorithm described in the previous subsection is proposed. The main difficulty remains in finding a good initialization point.

\section{Examples and Comparisons}

In this part, we show the efficiency of the method proposed above. Two examples will be discussed. For all these simulations, $l_{max}$ was set to $3$.

\subsection{First example: $\tr(A) = 0$}

For this first example, consider the following system:
\begin{equation} \label{eq:ex1}
	A = \left[ \begin{smallmatrix} 0.2 & 0 \\ 0.2 & -0.2 \end{smallmatrix} \right], \quad B = \left[ \begin{smallmatrix} -1 & 0 \\ -1 & -1 \end{smallmatrix} \right], \quad C = I_2.
\end{equation}

To initiate the algorithm, we choose $K = \left[ \begin{smallmatrix} 1.2 & 0 \\ -1 & 1.8 \end{smallmatrix} \right]$. The maximum delay has been computed using the algorithm presented above. The different strategies have been compared in Figure \ref{fig:ex1} to see the conservatism introduced by each of them. 
\begin{figure}
	\centering
	\includegraphics[width=9cm]{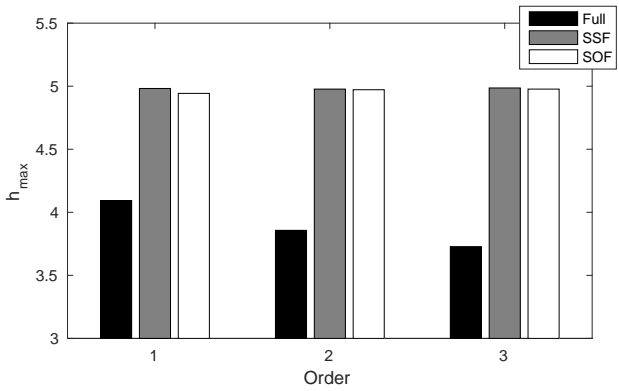}
	\caption{Maximum delay for which there exists a feedback gain which stabilizes system \eqref{eq:ex1}. In the legend, ``full'' corresponds to Lemma \ref{sec:finsler3} (without the use of the Jordan form), ``SSF'' is the application of Corollary \ref{sec:coroSSF} and ``SOF'' is the application of Corollary \ref{sec:coroSOF}.}
	\label{fig:ex1}
\end{figure}
It appears clearly that working on the Jordan form with a specific structure for $\mathcal{I}$ leads to better results ($25\%$ improvement). In this example of small dimension, there is no impact while increasing the order $N$, it even leads to poorer results for the case without any pretreatment. Then, there is no hierarchy in the stabilization even if there is one in the stability. The other conclusion which will be emphasized in the second example is the small conservatism introduced between Lemmas \ref{sec:coroSSF} and \ref{sec:coroSOF}. Unfortunately, there is no proof showing that the methodology of this paper always lead to better results but it appears on many examples that this is the case.

Table \ref{tab:ex1} is a comparison between different synthesis. The two last parts referred to Theorem \ref{sec:theo} with the use of Lemma 2. Two versions are used: $\varepsilon^1$ stands for $\varepsilon_1 = \varepsilon_2 = \varepsilon_3 = 1$ and all the others set to $0$, while $\varepsilon^2$ is the same values but with $\varepsilon_2 = 0.5$. In \cite{barreau:hal-01488920}, we developed another inequality, closely related to $\varepsilon^2, N = 1$. It was originally made for distributed delay but it can be adapted to deal with point-wise delay easily. $h_{max}^{spectral}$ is the maximal theoretical bound for stability for system \eqref{eq:modifiedProblem} with $A_d = BK$. This bound has been obtained using \cite{freq}.
\begin{table}[h]
	\centering
	\begin{tabular}{c|c|c|c|c}
		& Cor. \ref{sec:coroSSF} & $K$ & Thm. \ref{sec:theo} & $h_{max}^{spectral}$ \\
		\hline
		$N = 1$ & $4.982$ & $\left[ \begin{smallmatrix} 0.1979 & 0.0057 \\ -0.1195 & 0.0383 \end{smallmatrix} \right]$ & $4.986$ & $4.987$ \\
		$N = 2$ & $4.977$ & $\left[ \begin{smallmatrix} 0.2011 & 0.0001 \\ -0.1463 & 0.0915 \end{smallmatrix} \right]$ & $4.980$ & $4.980$ \\
		$N = 3$ & $4.986$ & $\left[ \begin{smallmatrix} 0.2005 & 0 \\ -0.1375 & 0.0744 \end{smallmatrix} \right]$ &  $4.991$ & $4.991$ \\
		\hhline{=|=|=|=|=}
		$\varepsilon^1$ & Lem. 2 & $K$ & Thm \ref{sec:theo} & $h_{max}^{spectral}$ \\
		\hline
		$N = 1$ & $1.184$ & $\left[ \begin{smallmatrix} 0.9674 & 0 \\ -31519 & 0.2178 \end{smallmatrix} \right]$ & $1.40$ & $1.43$ \\
		$N = 2$ & $1.186$ & $\left[ \begin{smallmatrix} 0.5802 & 0 \\ -0.1311 & 0.2064 \end{smallmatrix} \right]$ & $2.23$ & $2.23$ \\
		$N = 3$ & $1.186$ & $\left[ \begin{smallmatrix} 0.5817 & 0 \\ -0.1991 & 0.2582 \end{smallmatrix} \right]$ & $2.21$ & $2.23$ \\
		\hhline{=|=|=|=|=}
		$\varepsilon^2$ & Lem. 2 & $K$ & Thm \ref{sec:theo} & $h_{max}^{spectral}$ \\
		\hline
		$N = 1$ & $1.256$ & $\left[ \begin{smallmatrix} 0.6693 & 0 \\ -0.2440 & 0.2503 \end{smallmatrix} \right]$ & $1.89$ & $1.98$ \\
		$N = 2$ & $1.257$ & $\left[ \begin{smallmatrix} 0.6642 & 0 \\ -0.2603 & 0.2671 \end{smallmatrix} \right]$ & $1.89$ & $1.99$ \\
		$N = 3$ & $1.257$ & $\left[ \begin{smallmatrix} 0.6933 & 0 \\ 9.1431 & 0.2883 \end{smallmatrix} \right]$ &  $1.89$ & $1.92$ \\
	\end{tabular}
	\caption{Maximum delays for which it is possible to synthesize a static state feedback gain $K$ for system \eqref{eq:ex1}. }
	\label{tab:ex1}
	\vspace{-0.3cm}
\end{table}
\begin{table}[h]
	\centering
	\begin{tabular}{c|c|c|c|c|c|c}
		\multirow{2}*{Solver} & \multicolumn{2}{|c|}{$N = 1$} & \multicolumn{2}{|c|}{$N = 2$} & \multicolumn{2}{|c}{$N = 3$} \\
		& $h_{max}$ & \#It. & $h_{max}$ & \#It. & $h_{max}$ & \#It. \\
		\hline
		Cor. \ref{sec:coroSSF} & $4.98$ & $82$ & $4.97$ & $77$ & $4.98$ & $73$\\
		\hline
		Thm. \ref{sec:theo} / PenBMI & $5.5$ & $446$ & $5.71$ & $424$ & $5.2$ & $376$ \\
	\end{tabular}
	\caption{Maximum allowable delay $h_{max}$ and number of iterations \#It. for SSF of \eqref{eq:ex1} using Theorem \ref{sec:theo} and a BMI solver. The number of iteration is an average on each step of the algorithm. }
	\label{tab:penlab}
	\vspace{-0.3cm}
\end{table}
The naive version of \cite{barreau:hal-01488920} or random choices of $\varepsilon$ is roughly $4$ times less effective than the method developed here. The most important characteristic is the difference between the second and fourth columns. We also want to point out that there is a very tight bound between Corollary \ref{sec:coroSSF} and Theorem \ref{sec:theo}. That means the synthesized gain $K$ is at the boundary of the area of stability. In other words, there is nearly an equivalence between Corollary \ref{sec:coroSSF} and Theorem \ref{sec:theo} for this example. The choices for $\varepsilon$ induce a gap between these two values but the difference is small in the optimized case. The bound obtained by Theorem \ref{sec:theo} is also very close to the real upper bound for the delay meaning the stability area is correctly estimated. Notice that it does not imply a better synthesis but it is a first step. Even if there is not a significant difference between the orders, for large systems or systems with a distributed delay $\int_{t-h}^t x_t(s) ds$ there is a clear improvement as $N$ increases. Using BMIBNB from \cite{1393890} does not provide any solution and $h_{max} = 0$. \\
The lines about PENLAB in Table \ref{tab:penlab} are solutions to the BMI problem of Theorem \ref{sec:theo} implemented with a classical path-algorithm for increasing $h$ and the same initial condition. The efficiency of this BMI solver has been showed in \cite{henrion2005solving,kocvara2012pennon} for example. The results obtained are clearly better but it takes a lot more time because it is an NP-hard problem and the complexity is unlikely to be polynomial \cite{vanantwerp2000tutorial}. The number of iterations to converge is roughly five time higher with the BMI solver, showing that this solution cannot be easily adapted to larger system. Moreover, we didn't take into account the time needed to evaluate the hessian matrix and it is by far the most time-expensive part of the BMI solver. It is also known to be highly sensitive to its initial value. Moreover, the differences between Corollary \ref{sec:coroSSF} and PenLab at the same order are not that important.

%
%

\subsection{Second example: $B$ is not invertible}

This time, we consider the SSF only and the $B$ matrix is a single column:
\begin{equation} \label{eq:ex2}
	A = \left[ \begin{smallmatrix} 0 & 1 \\ -2 & -0.1 \end{smallmatrix} \right], \quad B = \left[ \begin{smallmatrix} 0 \\ 1 \end{smallmatrix} \right].
\end{equation}

With the starting point $K = \left[ \begin{smallmatrix} -1 & -5 \end{smallmatrix} \right]$ at $h = 0.1$, the maximum allowable delay is $h_{max} = 0.602$ with $K =  \left[ \begin{smallmatrix} -8.49 & 20.04 \end{smallmatrix} \right]$ for the order $N = 1$. Using PenLab and Theorem \ref{sec:theo}, we get $h_{max} = 1.678$ but it takes a longer timer and cannot be applied to large systems. Several conclusions can be drawn from this situation. First, it is indeed possible to deal with singular matrices. Considering $T$ is the matrix such that $T^{-1} A T$ is in its jordan normal form, numerically, $T^{-1} T$ should be $I_2$ which is not exactly the case. The process to transform $A$ in its Jordan form is numerically not robust and can lead to uncertainties. One can use another algorithm like Schur decomposition but it may not prevent the spreading of eigenvalues. And finally, if one sets $C = \left[ \begin{smallmatrix} -1 & -5 \end{smallmatrix} \right]$ and $K = 1$ as the initial point the algorithm does not provide any solution. So it is needed to have a full $X$ and not only a scalar.

\section{Conclusion}

As a conclusion, an efficient method using a reduced and optimized number of slack variables has been proposed to deal with the synthesis problem of a time-delay system. The comparison with a BMI solver shows a low complexity for similar results. This methodology ca be easily adapted to a large variety of problems: sampled systems, systems with distributed delays... Nevertheless, this approach may be a more generic using the tools developed in the book \cite{Svariable} about S-variables. 

\bibliographystyle{IEEEtran}
\bibliography{IEEEabrv,report_draft}

\end{document}